\documentclass[numbers,webpdf,imaiai]{ima-authoring-template}%

\graphicspath{{Fig/}}
\usepackage{cleveref}
\usepackage{subfigure}
\newtheorem{assumption}{Assumption}
\newtheorem{corollary}{Corollary}
\theoremstyle{thmstyletwo}%
\newtheorem{theorem}{Theorem}
\newtheorem{example}{Example}%

\newtheorem{definition}{Definition}

\numberwithin{equation}{section}

\begin{document}

\DOI{}
\copyrightyear{}
\vol{}
\pubyear{}
\access{Advance Access Publication Date: Day Month Year}
\appnotes{Paper}
\firstpage{1}


\title[On the Stability and Accuracy of Clenshaw-Curtis Collocation]{On the Stability and Accuracy of Clenshaw-Curtis Collocation}

\author{Ahmed Atallah*
\address{\orgdiv{Department of Mechanical and Aerospace Engineering}, \orgname{University of California San Diego},\\ \orgaddress{\street{9500 Gilman Drive, La Jolla}, \postcode{CA 92093}, 
\country{USA}}}}
\author{Ahmad Bani Younes 
\address{\orgdiv{Aerospace Engineering Department}, \orgname{San Diego State University},\\ \orgaddress{\street{5500 Campanile Drive, San Diego}, \postcode{CA 92182}, \country{USA}}}}


\corresp[*]{Corresponding author: \href{email:aatallah@eng.ucsd.edu}{aatallah@eng.ucsd.edu}}



\abstract{
We study the A-stability and accuracy characteristics of  Clenshaw-Curtis collocation. We present closed-form expressions to evaluate the Runge-Kutta coefficients of these methods. From the A-stability study, Clenshaw-Curtis methods are A-stable up to a high number of nodes. High accuracy is another benefit of these methods; numerical experiments demonstrate that they can match the accuracy of the Gauss-Legendre collocation, which has the optimal accuracy order of all Runge-Kutta methods.}


\maketitle

\section{Introduction}

An initial value problem (IVP) seeks the solution of $y(t)\in \mathbb{R}$ that satisfies the ordinary differential equation (ODE)
\begin{equation}
\label{e:GenFstOde}
\begin{aligned}
    \dot{y}\equiv\frac{\text{d}{y}}{\text{d}t} = f(t,y),
    \end{aligned}
\end{equation}
where $t_0\leq t\leq t_f$ and  $y(t_0)=y_0$. 
 The solution $y(t)$ exists and is unique for all $t\in [t_0,t_f]$ under the following assumption:
 \begin{assumption}
 \label{assump1}
The function $f(t,y)$ is continuous and satisfies a Lipschitz condition on the region $[t_0,t_f]\times\mathbb{R}$, i.e. there exists a constant $L>0$ such that
\begin{equation}|f(t,y)-f(t,x)|\leq L|x-y|
\end{equation}
for all $x,y\in \mathbb{R}$ and for all $t\in[t_0,t_f].$
 \end{assumption}
 
 Numerical methods typically divide the integration interval $[t_0,t_f]$ into a number of steps. A fixed step size, $h$, is assumed, i.e. $h=(t_f-t_0)/N$, for some integer $N\geq 1$ and let $t_1 = t_0+h,t_2=t_0+2h,\hdots, t_N=t_0+Nh=t_f $. We denote by $y_n$ a numerical approximant to the exact solution $y(t_n), n=0,1,\hdots,N$. The exact solution at $t_{n+1}$ can be obtained by integrating from $t_n$ to $t_{n+1}$:
 \begin{equation}
  \label{e:exact}
  \begin{aligned}
     y(t_{n+1}) &= y(t_n)+\int_{t_n}^{t_{n+1}}f(t,y)\text{d} t \\&= y(t_n)+h\int_{0}^{1}f(t_n+t h,y(t_n+ht))\text{d} t 
     \end{aligned}
 \end{equation}
 Runge-Kutta methods replace the integral in \cref{e:exact} by a linear combination of the function $f(t,y)$ evaluated at a number of nodes within the time step; that is, for an $s-$node Runge-Kutta method     \begin{equation}
    \label{e:IRK}
	\begin{aligned}    	
	  y_{n+1}=&y_n+h\sum_{j=1}^{s}b_jf(t_n+c_jh,Y_j),
\\
	Y_i=&y_n+h\sum_{j=1}^{s}a_{ij}f(t_n+c_jh,Y_j) \quad,\quad(i=1,2,...,s).   
	    \end{aligned}
	\end{equation}
The matrix $\mathbf{A}=(a_{ij})_{i,j=1,2,\hdots,s}$ is known as the Runge-Kutta matrix, while 
\begin{equation*}
    \begin{aligned}
    \mathbf{b} = \begin{bmatrix}b_1&b_2& \hdots & b_s
    \end{bmatrix}^T \quad \text{and}\quad 
    \mathbf{c} = \begin{bmatrix}c_1&c_2& \hdots & c_s
    \end{bmatrix}^T
    \end{aligned}
\end{equation*}
are vectors that contain the weights and nodes, respectively. (In this paper, we denote by the Runge-Kutta coefficients of a Runge-Kutta method the elements of its $\mathbf{A}$ matrix and $\mathbf{b}$ and $\mathbf{c}$ vectors.) Each Runge-Kutta method is described by its order of accuracy $p$. That is, if a method of order $p$ implies that $||y(t_1)-y_1||=O(h^{p+1})$ as $h\to 0$. Runge-Kutta methods are also classified according to
whether or not they are A-stable. For A-stable methods, if the continuous-time system in \cref{e:GenFstOde} is linear and asymptotically stable, then the discrete-time system \cref{e:IRK} is also asymptotically stable for all $h>0$; thus, A-stable methods are suitable for stiff problems.   

Runge-Kutta methods generally fall into two categories: explicit or implicit methods.
 A Runge-Kutta method is called explicit if $a_{ij}=0$ when $j\geq i$ and implicit otherwise. 
 Explicit Runge-Kutta (ERK) methods are not A-stable, and no ERK method has order $p>s$. 
 In contrast, some implicit Runge-Kutta (IRK) methods, especially those based on orthogonal polynomials, outperform explicit methods in terms of stability and accuracy. For example, Gauss-Legendre collocation methods, whose nodes are linearly related to the zeros of the orthogonal Legendre polynomials, are A-stable and optimal in the sense that the $s-$node method is of order $2s$. 
 These properties are behind the special attention these methods have received since the inception of IRK methods by Kuntzmann \cite{Kuntzmann1961implicit} and later by Butcher \cite{butcher1964implicit}.
  
 Clenshaw-Curtis collocation, a family of collocation methods whose nodes are based on Chebyshev points \cite{trefethen2008gauss}, is less celebrated than Gauss-Legendre collocation. They are, however, posses several advantages over Gauss-Legendre methods for practical implementations. First, as introduced in this paper, the Runge-Kutta coefficients of the Clenshaw-Curtis collocation can be evaluated using explicit formulas, eliminating the need to precompute and store these coefficients. Second, these methods are adaptive in the sense that the nodes of the $s-$node method are included in the nodes of the $(2s+1)$-node method; thus, if a given accuracy is not achieved for an $s-$node formula, the number of nodes can be increased to $2s+1$ without re-evaluating the ODE at the original $s$ nodes. Third, as Butcher \cite{butcher1992role} points out, the Clenshaw-Curtis collocation allows for accurate extrapolation of the results computed in one step to obtain starting values for the iteration of the nodes` values in the next step.

Clenshaw-Curtis methods are symmetric collocation methods that include the ends of the interval $[0,1]$. Then it follows from \cite{barrio1999characterization} that they are at A-stable when $s\leq 7$. In this paper, it is proved the A-stability of these methods when $s\leq 78$.

Collocation methods have the property that any $s$-node collocation method is at least of order $s$. Vigo-Aguiar and Ramos \cite{vigo2007family} proved that when $s$ is odd, the order of the $s$-node Clenshaw-Curtis method is $s+1$. As a result, these methods fall short of the optimal order attained by Gauss-Legendre collocation methods. Results of numerical tests \cite{woollands2019nonlinear,atallah2020accuracy}, however, show that Clenshaw-Curtis collocation-based integrators are more efficient than the current state-of-the-art integration methods and have accuracy levels comparable with Gauss-Legendre collocation-based integrators. The question that naturally arises is then: What accounts for the high accuracy of Clenshaw-Curtis collocation methods despite their low orders of accuracy? The reason is that the order of numerical methods is not always able to anticipate the actual accuracy of collocation methods. We present numerical examples that show that Clenshaw-Curtis collocation can be as accurate as Gauss-Legendre collocation and more accurate than a family of collocation methods with the same accuracy order, namely Newton-Cotes collocation \cite{shampine1972stable}.

 The remainder of the paper is structured as follows. \Cref{sec:chebysevpolynomials} provides a summary of the properties of Chebyshev polynomials, points, and approximations. \Cref{sec:CC-IRK} is devoted to deriving the explicit formulas for the coefficients of Clenshaw-Curtis collocation.
 \Cref{sec:stab} looks into the stability of Clenshaw-Curtis collocation methods, while \Cref{sec:accuracy} studies their accuracy. Concluding remarks are made in \cref{sec:conclusions}.

\section{Background Material}
\label{sec:chebysevpolynomials}
\subsection{Chebyshev polynomials}
Let $\mathcal{N}=\{0,1,\hdots\}$ or $\mathcal{N}=\{0,1,\hdots,N\}$ be an index set for a finite nonnegative integer $N$, $(a,b)$ be an interval in $\mathbb{R}$, and $\alpha$ denote a positive measure on $(a,b)$. 
\begin{definition} A system of polynomials $\{p_k(\xi)=\sum_{i=0}^k \beta_i \xi^i, i\in \mathcal{N}, \beta_i\neq 0\}$ is called an \textit{orthogonal system of polynomials} over the interval $(a,b)$ with respect to the measure $\alpha$ if the following relations hold: 
\begin{equation}
\label{e:orthog}
\begin{aligned}
    \int_a^b p_k(\xi) p_j(\xi)\text{d} \alpha(\xi)=\gamma_k \delta_{kj},\quad k,j \in \mathcal{N},
    \end{aligned}
\end{equation}
where $\delta_{kj}=0$ if $k\neq j$ and $\delta_{kj}=1$ if $k=j$ and 
\begin{equation*}
\begin{aligned}
    \gamma_k =\int_a^b p_k^2 \text{d} \alpha(\xi),\quad k\in \mathcal{N}.
    \end{aligned}
\end{equation*}
are called the \textit{normalization constants}.
\end{definition}
The measure $\alpha$ usually has a density $w(\xi)$  or is a discrete measure with weight $w_i$ at the point $\xi_i\in (a,b)$. For the former case, the orthogonality relations in \cref{e:orthog} become
\begin{equation}
    \int_{a}^b p_k(\xi) p_j(\xi) w(\xi) \text{d} \xi=\gamma_k \delta_{kj},\quad k,j \in \mathcal{N},
\end{equation}
and for the latter case, become
\begin{equation}
\label{e:discrete_orth}
    \sum_i p_k(\xi_i) p_j(\xi_i) w_i =\gamma_k \delta_{kj},\quad k,j \in \mathcal{N}.
\end{equation}  
For the case of discrete orthogonality, neither the nodes' locations nor the weights are arbitrary but are associated with the particular choice of polynomials to be compatible with the orthogonality conditions of \cref{e:discrete_orth}. Generally, the interior nodes are associated with the zeroes or the extrema of the particular chosen set of orthogonal polynomials.

On the interval $(-1,1)$, Legendre polynomials, $P_k (\xi)$, Chebyshev polynomials of the first kind, $T_k(\xi)$, and Chebyshev polynomials of the second kind, $U_k(\xi)$, are classical orthogonal systems of polynomials with respect to the measures $\text{d} \xi$, ${1}/{\sqrt{1-\xi^2}}\text{d} \xi$, and $\sqrt{1-\xi^2}\text{d} \xi$, respectively. Every orthogonal system of polynomials satisfies three-term recurrence relations. These classical polynomials obey the recurrence relations 
 \begin{equation*}
 \begin{split}
 P_{k+1}(\xi)&= \frac{2k+1}{k+1}\xi P_k(\xi)- \frac{k}{k+1}P_{k-1}(\xi),\\
    T_{k+1}(\xi) &= 2 \xi T_k(\xi) -T_{k-1}(\xi),
     \\
     U_{k+1}(\xi)& = 2 \xi U_k(\xi) -U_{k-1}(\xi),
      \end{split}
 \end{equation*}
 with $P_0=T_0(\xi)=U_0(\xi)=1$, $P_1=T_1(\xi)=\xi$, and $U_1(\xi) = 2\xi$.
 Chebyshev polynomials of the first and second kind can also be defined by the trigonometric relations%
  \begin{equation}
  \label{e:trig}
  \begin{aligned}
  T_k(\xi)=\cos(k\theta),\quad
 U_k(\xi)= \sin\big((k+1)\theta\big)/\sin{\theta}
  \end{aligned}
  \end{equation}
when $\xi=\cos \theta$. From these relations, the following relation can be derived
\begin{equation}
\label{e:firstAndSecondType}
    T'_k(\xi)=kU_{k-1}(\xi).
\end{equation}

An important property of the first-kind Chebyshev polynomials is that when $i\geq 2$, the integral of $T_i(\xi)$ is function of $T_{i+1}$ and $T_{i-1}$:
\begin{equation}
\begin{split}
\int T_k(\xi)=&\frac{1}{2}\bigg (\frac{T_{k+1}(\xi)}{i+1}-\frac{T_{k-1}(\xi)}{k-1}\bigg).
\end{split}
\end{equation}
and
 \begin{equation}
\label{e:integration}
\begin{aligned}
\int T_0(\xi)=&T_1(\xi),&
\int T_1(\xi)=&\frac{1}{4}T_0(\xi)-\frac{1}{4}T_2(\xi).
\end{aligned}
\end{equation}

\subsection{Chebyshev points}
Orthogonal polynomials have the property that if $\{p_k\}$ is a system of orthogonal polynomials over the interval $(a,b)$, all the $k$ zeros of the polynomial $p_k$ are simple and reside in the interval $(a,b)$ \cite{iserles2009first}. The nodes of Gauss-Legendre collocation are the zeros of Legendre polynomials, mapped from the orthogonality interval $(-1,1)$ to the interval $(-1,1)$. These nodes cannot be evaluated via closed-form expressions. In contrast, the trigonometric definition \cref{e:trig} allows deriving explicit formulas for the zeros of Chebyshev polynomials of the first and second kinds. Chebyshev points are the zeros of Chebyshev polynomials of the second kind in addition to the boundary nodes ($-1$ and $+1$). Therefore, for $s\geq 2$, the $s-$set of Chebyshev points are the zeros of the monic polynomial
\begin{equation}
\label{e:zeros-11}
   q_s(\xi) =  \frac{(\xi^2-1)}{2^{s-2}}U_{s-2}(\xi)
\end{equation}
and from \cref{e:firstAndSecondType}
\begin{equation}
\label{e:zeros-12}
   q_s(\xi) =  \frac{(\xi^2-1)}{(s-1)2^{s-2}}T'_{s-1}(\xi).
\end{equation}
In other words, Chebyshev points are the extrema of Chebyshev polynomials of the first kind on the interval $[-1,1]$. From the trigonometric relations in \cref{e:trig}, these points are given by
 \begin{equation}
  \label{e:CGLNodes}
  \begin{aligned}
   \xi_i = -\cos \frac{i-1}{s-1}\pi, \quad(i=1,\hdots,s).
  \end{aligned}
  \end{equation}
  The negative sign in \cref{e:CGLNodes} is arbitrary and not universally chosen, but it offers the heuristic advantage that $i=1$ generates $\xi_1 = -1$, the left end of the boundary and $\xi_s=1$ is the right end of the (-1,1).
  
  Chebyshev polynomials of the first kind are discrete orthogonal for Chebyshev points, with the discrete orthogonality relations being
   \begin{equation}
  \label{e:innerProduct}
  \begin{aligned}
  \sum_{i=1}^{s} {}^{''} T_{k-1}( {\xi}_i)T_{j-1}( {\xi}_i)
  =
   \begin{cases}
  0,\quad\text{if }k \neq j, \\
  s-1,\quad \text{if }k=j=1, \\
 (s-1)/{2},\quad \text{if }k=2=1,2,...,s-1, \\
  s-1,\quad \text{if }k=j=s.
  \end{cases} 
  \end{aligned}
  \end{equation}
 (The ${}^{''}$ on the summation implies that both the first and last terms in the summation are halved.) Therefore, as we shall see in the following subsection, there is no need for matrix inversion to obtain the corresponding coefficients when used to approximate a function.

 At the boundary points $\pm1$The $j$th derivative of the $s$th Chebyshev polynomial is given by
 \begin{equation}
 \label{e:derivatives}
    T^{(j)}_{s}(\pm 1)=(\pm 1)^{s+j}\prod_{k=0}^{j-1}\frac{s^2-k^2}{2k+1}.
\end{equation}

\subsection{Chebyshev Approximations}
\label{subsec:approx}
A Lipschitz continuous function defined on the interval $[-1,1]$ has an absolutely convergent Chebyshev series \cite{trefethen2017multivariate}, i.e.,
  \begin{equation}
  \label{e:ChebSeries}
      f(\xi) = \sum_{k=1}^{\infty} {}^{'}  \alpha_{k-1} T_{k-1}(\xi).
  \end{equation}
  (The ${}^{'}$ on the summation means the first term of the series is halved.) Due to the orthogonality of Chebyshev polynomials with respect to the measure $1/{\sqrt{1-\xi^2}}\text{d}\xi$, the coefficients ($\alpha_i$) of this series are then given by
  \begin{equation}
  \label{e:coefficients}
  \begin{aligned}
      \alpha_k = 
        \frac{2}{\pi}\int_{-1}^1 T_{k}(\xi) f(\xi)\frac{1}{\sqrt{1-\xi^2}}\text{d}\xi.
      \end{aligned}
  \end{equation}
  The {Chebyshev projection}, $S_s(\xi)$, is a polynomial approximant resulting from truncating the Chebyshev series in \cref{e:ChebSeries} at the $s$th term \cite{townsend2014computing}, i.e.,
  \begin{equation}
      S_s(\xi) = \sum_{i=1}^{s} {}^{'} \alpha_{k-1} T_{k-1}(\xi).
  \end{equation}
  
  To avoid evaluating the integral in \cref{e:coefficients}, another polynomial approximation of the function $f$, known as Chebyshev interpolant, is obtained by interpolation in Chebyshev points. The Chebyshev interpolant with $s$ Chebyshev points, $f_s(\xi)$, is given by 
  \begin{equation}
      f_s(\xi) = \sum_{k=1}^s {}^{''} \beta_{k-1} T_{k-1}(\xi),
  \end{equation}
  where the coefficients $\{\beta_i\}$ satisfy the discrete orthogonality conditions of \cref{e:innerProduct}; therefore 
  \begin{equation}
  \label{e:coeffs}
     \beta_k = \frac{2}{s-1}\sum_{j=1}^{s} {}^{''} f ( {\xi}_j)T_{k}( {\xi}_j).
  \end{equation}
  A primary advantage of the interpolation in Chebyshev points is that the interpolation error is within a factor of 2 of the truncation error \cite{trefethen2019approximation}. This fact, combined with the absolute convergence of the Chebyshev series, leads to the conclusion that the interpolation error, $f-S_s=\sum_{i=s+1}^\infty \alpha_i T_i(\xi)$, converges uniformly and absolutely to $0$. In addition, the smoother the function is, the faster the interpolation error converges as $s \to \infty$. In particular, if the function $f$ is analytic, the interpolation errors converge to 0 geometrically \cite{Boyd2001}.
  
The convergence properties of Chebyshev interpolants are inherited to integrating over a bounded interval while using a Chebyshev interpolant to approximate the integrand. The family of methods that results when $f(t,y) = f(t)$ is known as Clenshaw-Curtis quadrature. Thus Clenshaw-Curtis collocation is the name we have given to the family of collocation methods that employs Chebyshev interpolants to approximate the integrand for general ODEs.


 \section{Construction of the methods}
\label{sec:CC-IRK} 
\begin{definition}
\label{def:collocation}
For $s$ a positive integer and $c_1,\hdots,c_s$ distinct real numbers (typically between 0 and 1), the \textit{collocation method} consists of finding the polynomial $u(t_n+\tau h)$ of degree $s$ such that 
\begin{subequations}
\label{e:conditions}
\begin{equation}
\label{e:conditions1}
\begin{split}
    u(t_n)&=y_n,
    \end{split}
    \end{equation}
    \begin{equation}
    \label{e:conditions2}
    \begin{aligned}
    \dot{u}(t_n+c_ih) = f(t_n+c_ih,u(t_n+c_ih)),\quad i=1,\hdots,s.
    \end{aligned}
  \end{equation}
\end{subequations}
The numerical solution at $t_{n+1}$ is given by
\begin{equation}
    y_{n+1} = u(t_n+h).
\end{equation}
\end{definition}
Collocation methods are IRK methods; thus, they take the form \ref{e:IRK}. For a sequence of numbers $\{c_i\}_{i=1}^s\in [0,1
    ]$, the coefficients $\{a_{ij}\}$ and $\{b_j\}$ of the resulting collocation method are typically given by 
     \begin{equation}
    \begin{aligned}
        a_{ij}= \int_{0}^{c_i}l_j(t) \text{d}t,\quad b_j =\int_0^1 l_j(t) \text{d}t\quad i,j=1,\hdots,s,
        \end{aligned}
    \end{equation}
    where the $l_j(t)$ are the Lagrange polynomials
    \begin{equation}
        l_j(t) = \prod_{k\neq j}\frac{t-c_k}{c_j-c_k}.
    \end{equation}
Therefore, explicit formulas for the RK coefficients are not generally available. An exception is the coefficients of Clenshaw-Curtis collocation methods. 
	\begin{theorem}
	 For the $s-$node Clenshaw-Curtis collocation method, the RK coefficients are given by 
	 \begin{subequations}
	    \begin{equation}
    \label{e:coeffs2}
    \begin{split}c_i& = \frac{1}{2}(1-\cos ((i-1)\pi/(s-1))
    \end{split}
    \end{equation}
     \begin{equation}
        a_{ij}=\frac{1}{2-\delta_{1j}-\delta_{sj}} \frac{1}{2(s-1)}\sum_{k=1}^s{}^{''}\cos (-(k-1)(j-1)\pi/(s-1))I_{i,k-1}
        \end{equation}
        \begin{equation}
        \begin{split}
       b_j = a_{sj}  
    \quad j =1,2,..,s
    \end{split}
    \end{equation}
    \end{subequations}
     where 
     \begin{equation*}
	\begin{aligned}
			I_{i0}=2 c_i,\quad I_{i1}=\frac{1}{4}\bigg(\cos \frac{-2(i-1)\pi}{s-1}-1\bigg) 
	    \end{aligned}
	    \end{equation*}
     and\text{ }for $k> 1$,
    \begin{equation*}
	    I_{ik}=
	    \frac{1}{2(k+1)}\bigg(\cos \frac{-(k+1)(i-1)\pi}{s-1}+(-1)^k\bigg)-\frac{1}{2(k-1)}\bigg(\cos \frac{-(k-1)(i-1)\pi}{s-1}+(-1)^k\bigg).
	\end{equation*}
     	\end{theorem}

\begin{proof}
The nodes of Clenshaw-Curtis collocation are Chebyshev points, transplanted from $[-1,1]$ to $[0,1]$, i.e.,
\begin{equation}
     c_i = \frac{1}{2}(1+\xi_i). 
\end{equation}
To satisfy the condition in \cref{e:conditions}, let $u$ be a polynomial of degree $s$ over the interval $[t_n,t_n+h]$ such that $u(t_n) = y_n$ and its derivative is a Chebyshev interpolant of degree $s-1$, i.e., 
\begin{equation}
\label{e:derive_u}
    \dot{u}(t_n+h(1+\xi)/2)= \sum_{k=1}^s {}^{''}\beta_{k-1}T_{k-1}(\xi)
\end{equation}
From \cref{e:coeffs},
\begin{equation}
\label{e:beta_eqn}
    \beta_k = \frac{1}{s-1}\sum_{j=1}^s{}^{''} f(t_n+c_jh,u(t_n+c_jh))T_{k}(\xi_j).
\end{equation}
Having satisfied \ref{e:conditions1} and \ref{e:conditions2}, then 
\begin{equation}
\label{e:yn+1}
    y_{n+1} =  u(t_n+h) = y_n+h\int_{0}^{c_i}\dot{u}(t_n+\tau h) \text{d} \tau= y_n+h/2\int_{-1}^{\xi_i} \dot{u}(t_n+\ h/2(1+\xi)) \text{d} \xi.
\end{equation}
Substituting \cref{e:derive_u} and \cref{e:beta_eqn} in \cref{e:yn+1} gives 
\begin{equation}
\label{e:yn+2}
    y_{n+1} = 
    y_n+h\sum_{j=1}^s{}^{''}f(t_n+c_jh,u(t_n+c_jh))\frac{1}{2(s-1)}\sum_{k=1}^s{}^{''}{T}_{k-1}(\xi_j)\int_{-1}^{1} {T}_{k-1}(\xi) \text{d} \xi.
\end{equation}
The values of the polynomial $u$ at the collocation nodes $u(t_n+c_ih)$, $i=1,\hdots, s$ are needed to evaluate $y_{t_n+1}$. They are given by
\begin{equation}
\begin{split}
    u(t_n+c_ih)&= y_n+h\int_{0}^{c_i}\dot{u}(t_n+\tau h) \text{d} \tau= y_n+h/2\int_{-1}^{\xi_i} \dot{u}(t_n+\ h/2(1+\xi)) \text{d} \xi.
    \end{split}
\end{equation}
Substituting \cref{e:derive_u} and \cref{e:beta_eqn} gives 
\begin{equation}
    \begin{split}
    u(t_n+c_ih)&= y_n+h\sum_{j=1}^s{}^{''}f(t_n+c_jh,u(t_n+c_jh))\frac{1}{2(s-1)}\sum_{k=1}^s{}^{''}{T}_{k-1}(\xi_j)\int_{-1}^{\xi_i} {T}_{k-1}(\xi) \text{d} \xi
    \end{split}
    \end{equation}
    By letting 
    \begin{equation}
     a_{ij}=\frac{1}{1-\delta_{1j}-\delta_{sj}} \frac{1}{2(s-1)}\sum_{k=1}^s{T}_{k-1}(\xi_j)\int_{-1}^{\xi_i} {T}_{k-1}(\xi) \text{d} \xi,
\end{equation}

    \begin{equation}
    \begin{split}
  Y_i= u(t_n+c_ih) = y_n+h\sum_{j=1}^sa_{ij}f(t_n+c_jh,u(t_n+c_jh)).
    \end{split}
\end{equation}
\cref{e:coeffs2} follows from \cref{e:trig}, \cref{e:CGLNodes}, and the integration property in \cref{e:integration}. Since $c_s=1$, $b_j= a_{sj}$.
	\end{proof}

  \section{Stability of the methods}
  \label{sec:stab}
  A standard problem to study the stability of numerical methods is the linear ODE
  \begin{equation}
  \label{e:linear}
  \begin{aligned}
      \dot{y} = \lambda y,\quad \lambda \in \mathbb{C}.      \end{aligned}
  \end{equation}
  The continuous-time system in \cref{e:linear} is asymptotically stable, i.e. $\lim_{t \to  \infty}y(t) =0$, if only if $Re(\lambda)<0$. The domain of linear stability of a numerical method is the set of all $z=h\lambda \in \mathbb{C}$ such that $\lim_{n \to \infty} y_n =0$. A numerical  method is said to be A-stable if its domain of linear stability satisfies
 \begin{equation}
     \mathbb{C}^{-}:= \{z\in \mathbb{C}|Re(z) < 0 \}\subseteq D,  \end{equation}
 and $A_0$ if $D$ satisfies
  \begin{equation}
     \mathbb{R}^{-}:= \{z\in \mathbb{R} | z < 0 \}\subseteq D.
 \end{equation}
 
 Applying a Runge-Kutta method to \cref{e:linear} results in the discrete-time system given by
 \begin{equation}
  \label{e:linear_discrete}
      y_{n+1}= r(h\lambda)y_n.
  \end{equation}
  Here $r(z)$ is the stability function defined by
    \begin{equation}
    \label{e:stability_function}
     r(z)=[1+z\pmb{b}^T(I-z A)^{-1}\pmb{1}].
  \end{equation}
  where $\pmb{1}= [1,1,\hdots,1]^T$. Since the necessary and sufficient condition for $\lim_{n \to \infty} y_n =0$ in \cref{e:linear_discrete} is that $|r(z)|<1$, we can make the following definition. 
 
\begin{definition}
\label{def:stability}
A Runge-Kutta method is called A-stable if and only if the stability function $r(z)$ satisfies 
 \begin{equation}
 \label{e:stability}
     |r(z)|<1.
 \end{equation}
 for all $z\in \mathbb{C}$ and $\text{Re}(z)<0$ and if the stability function satisfies the inequality in \cref{e:stability} when $z\in \mathbb{R}$ and $z<0$, the method is called A$_0-$stable.
\end{definition} 

For explicit Runge-Kutta methods, the stability function $r(z)$ is a polynomial; therefore, the condition of A-stability cannot be satisfied. In contrast, since the stability function is rational, some implicit Runge-Kutta methods are A-stable. For collocation methods, it is given by the following theorem.
\begin{theorem}[\cite{wanner1996solving}] The stability function of the collocation method based on the points $c_1,c_2,\hdots,c_s$ is given by
\begin{equation}
    \label{e:stability2}
    r(z) =N_s(z)/D_s(z)
    \end{equation}
where $D_s(z)$ and $N_s(z)$ are given by
\begin{equation}
    \begin{split}
        D_s(z) = \sum_{j=0}^s M^{(s-j)}(0)z^{j},\quad
        N_S(z) = \sum_{j=0}^s M^{(s-j)}(1)z^{j},
    \end{split}
\end{equation}
with
 \begin{equation}
 \label{e:zeros01}
     M_s(\tau) = \frac{1}{s!}\prod_{i=1}^s(\tau-c_i).  
 \end{equation}
  \end{theorem}
 Employing this theorem, the following corollary gives the stability function for Clenshaw-Curtis collocation.
  \begin{corollary}
  The polynomials $N_S(z)$ and $D_S(z)$ in \cref{e:stability2} for Clenshaw-Curtis collocation are given by
\begin{equation}
    \begin{split}
        D_S(z) = \sum_{j=0}^{s-1} (-1)^{s-j}d_jz^{j},\quad
        N_S(z) = \sum_{j=0}^{s-1} d_jz^{j},
    \end{split}
\end{equation}
  where $d_j\geq 0$ and are defined by
  \begin{equation}
  \label{e:d_j}
   d_j = \frac{2^j(s-j)}{2^{2s-2}(s-1)s!} \bigg((s-j-1)\prod_{k=0}^{s-j-1}\frac{(s-1)^2-k^2}{2k+1}+2\prod_{k=0}^{s-j}\frac{(s-1)^2-k^2}{2k+1}\bigg)
  \end{equation}
  \end{corollary}
 \begin{proof}
  From \cref{e:zeros-11,e:zeros01,e:zeros01} 
  \begin{equation}
    M_s((\xi+1)/2) = \frac{1}{2^ss!}q_s(\xi) = \frac{1}{2^s} \frac{1}{2^{s-2}(s-1)} (\xi^2-1){T'_{s-1}(\xi)}.
    \end{equation}
    The $j$th derivative of $M_s$ is then given by
      \begin{equation}
      \begin{split}
    M_s^{(j)}((\xi+1)/2) = &\frac{2^j}{2^{2s-2}(s-1)s!} \bigg(j(j-1)T_{s-1}^{(j-1)}(\xi)+\\&2j\xi T_{s-1}^{(j)}(\xi)+(\xi^2-1)T^{(j+1)}_{s-1}(\xi)\bigg).
    \end{split}
\end{equation}
At $\tau=1$
 \begin{equation}
  \begin{aligned}
  \label{e:M_sj}
    M_s^{(j)}(1) = \frac{2^j}{2^{2s-2}(s-1)} \bigg(j(j-1)T^{(j-1)}_{s-1}(1)+2jT_{s-1}^{(j)}(1)\bigg)\\
        \end{aligned}
\end{equation}
and at $\tau=0$
 \begin{equation}
  \begin{aligned}
      M_s^{(j)}(0) = \frac{2^j}{2^{2s-1}(s-1)} \bigg(j(j-1)T^{(j-1)}_{s-1}(-1)-2jT_{s-1}^{(j)}(-1)\bigg)
    \end{aligned}
\end{equation}
From \cref{e:derivatives}
 \begin{equation}
  \begin{aligned}
      M_s^{(j)}(0) = (-1)^{s-j}M_s^{(j)}(1).
    \end{aligned}
\end{equation}
Let 
\begin{equation}
\label{e:d}
   d_j= M^{(s-j)}(1).
\end{equation}
Then the polynomials in $D_s(z)$ and $N_S(z)$ in \cref{e:stability2}  are given by
\begin{equation}
    \begin{split}
        D_s(z) = \sum_{j=0}^{s-1} (-1)^{j}d_jz^{j},\quad
        N_S(z) = \sum_{j=0}^{s-1} d_jz^{j},
    \end{split}
\end{equation}
Substituting \cref{e:derivatives,e:M_sj} in \cref{e:d} gives \cref{e:d_j}.
   \end{proof}
A direct consequence of this corollary is to prove the A$_0$-stability for Clenshaw-Curtis collocation. 
\begin{theorem}\label{thm:stability}
  Clenshaw-Curtis collocation methods are A$_0$-stable. 
\end{theorem}
\begin{proof}
Suppose that $z = -x$ where $x\in \mathbb{R}$ and $x>0$,
\begin{equation}
\begin{split}
|N_S(z)| &= \bigg|\big|\sum_{j=0,j\text{ is even }}^{s-1} d_j x^j\big|-\big|\sum_{j=0,j\text{ is odd }}^{s-1} d_j x^j\big|\bigg|,\\
|D_s(z)|&=\bigg|\sum_{j=0}^{s-1} d_j x^j\bigg|.
\end{split}
\end{equation}
Therefore $|N_S(z)|<|D_s(z)|$ and hence $|r(z)|<1$ for all $z\in \mathbb{R}^{-}$. This completes the proof.
\end{proof}
Unlike A$_0$-stability, A-stability proof is not straightforward by just applying \cref{def:stability}. Alternatively, the following theorem can be used. 
\begin{theorem}
\label{thm:stability_collocation}
A Runge-Kutta method is $A$-stable if and only if
\begin{equation}
    |r(iy)|\leq1 \text{    for all real }y
\end{equation}
and
\begin{equation}
\label{e:secondcondition}
    r(z)\text{     is analytic for all $z$ in the left half plane}.
\end{equation}
\end{theorem}
Then we can prove the following.
\begin{theorem}
The Clenshaw-Curtis type collocation method with a number of nodes less than or equal to 78 is A-stable.
\end{theorem}
\begin{proof}
The first condition in \cref{thm:stability_collocation} is satisfied since $D_s(z) = N_S(-z)$. The second condition \cref{e:secondcondition} is satisfied if and only if the roots of $D_S(z)$ are in the right half-plane. Define by $R_s$ the set of the real parts of $D_s$'s roots. Therefore, the $s$-node Clenshaw-Curtis collocation is A-stable if and only if $R_s$ is a positive set. \Cref{fig:stability} shows minimum value of the set $R_s$. From the figure, all the roots are in the right half-plane when $s\leq 78$, which completes the proof.  
\end{proof}
\begin{figure}[b]
  \centering
  \includegraphics[width=10cm]{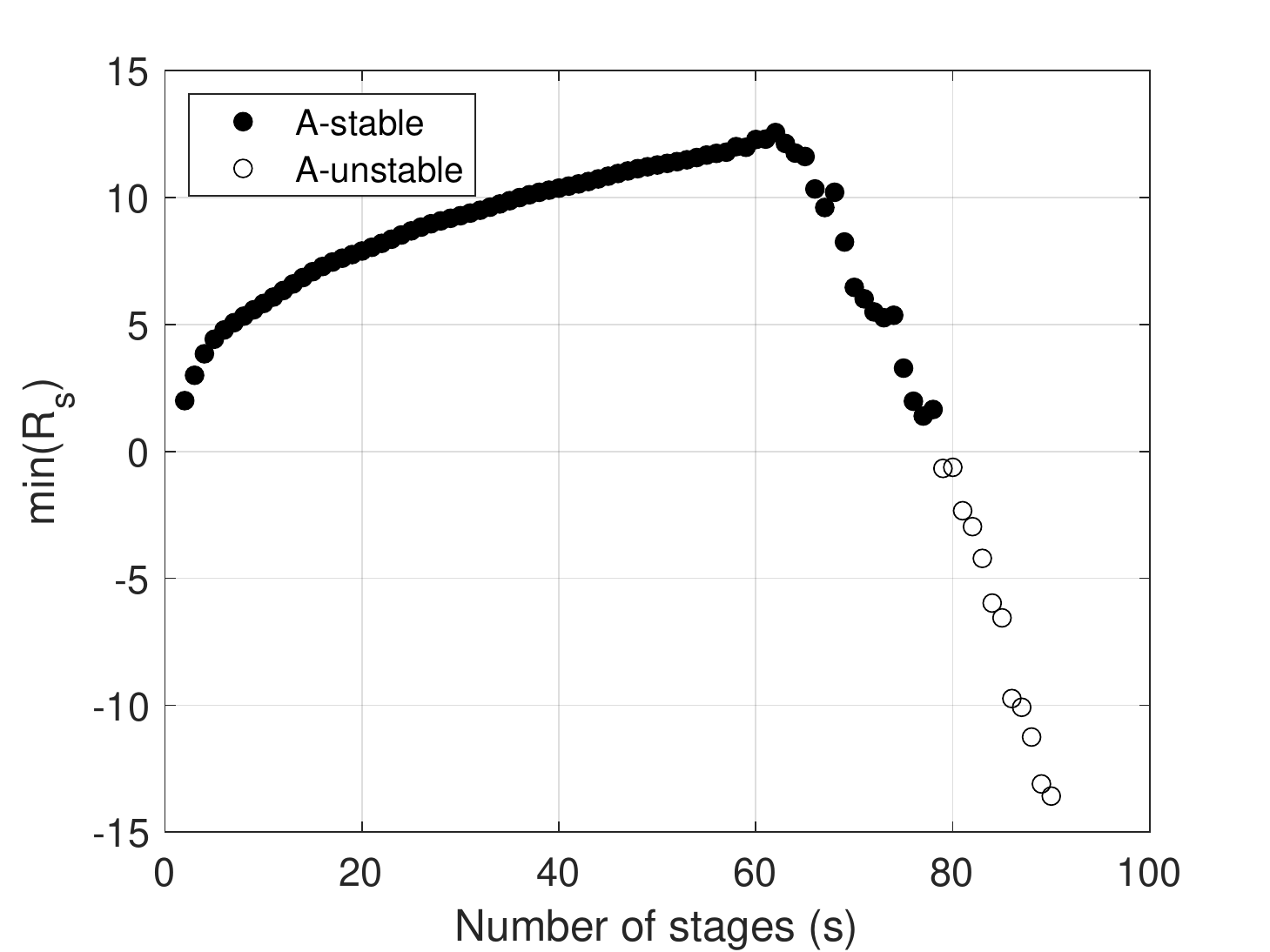}
  \caption{A-stability of Clenshaw-Curtis collocation}
  \label{fig:stability}
\end{figure}

\section{Accuracy of the methods} 
\label{sec:accuracy}
A Runge-Kutta method has order $p$ if for sufficiently smooth functions $
  ||y(t_0+h)-y_1 )||\leq Kh^{(p+1)}$, i.e., the Taylor series for $y(t_0+h)$ and $y_1$ coincide up to the term $h^p$. Therefore, the order of a Runge-Kutta method can be found by comparing the Taylor series expansion of the exact and the numerical solution. Fortunately, the order of collocation methods can be easily obtained in a more direct way.
\begin{theorem}[\cite{iserles2009first}]
\label{thm:accuracy}
  Let $M_s(t)$ be defined as in \cref{e:zeros01} and suppose that $M_s(t)$ is orthogonal to polynomials of degree $m-1$,
  \begin{equation}
  \begin{aligned}
      \int_{0}^1 M_s(\tau)\tau^j \text{d} \tau =0, \quad j=0,1,\hdots,m-1,
         \end{aligned}
  \end{equation}
  for some $m\in \{0,1,\hdots, s\}$. Then the collocation method is of order $s+m$. 
  \end{theorem}
This theorem implies that any $s-$node collocation method is at least of order $p=s$. It also follows that the $s-$node Gauss-Legendre collocation method has order $p=2s$ since the Legendre polynomial of degree $s$ is orthogonal for all polynomials of degree $<s$.

For Clenshaw-Curtis collocation, employing \cref{thm:accuracy}, Vigo-Aguiar and Ramos \cite{vigo2007family} proved the following:
 \begin{corollary}
When $s$ is odd, the Clenshaw-Curtis collocation has order $s+1$. 
 \end{corollary}
Clenshaw-Curtis collocation has the lowest order any symmetric collocation method can have. The family of Netwon-Cotes collocation methods is one example (see the proof in \cite
{shampine1972stable}). Using the order of accuracy implies the following: higher-order collocation methods are more accurate than lower-order ones; accuracy improves with increasing the number of nodes for each family of collocation methods; and families of the same order are almost accurate. The two examples below demonstrate that these implications are not always realized.

 \begin{example}
 \begin{equation}
 \begin{aligned}
     \dot{y} = y,\quad y(0)=1
     \end{aligned}
 \end{equation}
with solution $
     y(t) = e^t$ and $t_f=1$.
 \end{example}
\begin{example}
\begin{equation}
\begin{aligned}
\dot{y}=2y/t^3,\quad y(0) = 1
\end{aligned}
\end{equation}
with solution $
    y(t) = e^{1-1/t^2}$ and $t_f=3$.
\end{example}

The integration errors for the two examples above when integrating using the Gauss-Legendre, Clenshaw-Curtis, and Newton-Cotes collocation methods are shown in \Cref{fig:GLvsCC}. We used one integration step in both examples, i.e., $h=t_f-t_0$. To solve the resulting nonlinear algebraic system, a fixed-point iteration is used with a tolerance of $10-14$ and a maximum number of iterations of 100.
The integration error is measured by the absolute difference between the exact solution at $t_f$ and its approximant $y_1$.

In both examples, the Newton-Cotes decrease integration errors geometrically before a critical number of nodes before increasing geometrically. We can also see that Clenshaw-Curtis outperforms Newton-Cotes even before reaching that critical number. This is more evident in the second example. For example, the 30-node Clenshaw-Curtis method has an error of order $10^{-15}$, while the error of the 30-node Newton-Cotes collocation is of order $10^{-10}$.  

Unlike Newton-Cotes collocation methods, both Gauss-Legendre and Clenshaw-Curtis collocation methods are convergent. In the first example, the convergence rate of the Clenshaw-Curtis collocation is half that of the Gauss-Legendre collocation, as expected from their orders of accuracy. In contrast, in the second example, Gauss-Legendre and Clenshaw-Curtis methods have the same convergence rate before a critical number of nodes. After that number, the Clenshaw-Curtis convergence rate is, as expected, half that of Gauss-Legendre. In general, Gauss-Legendre is never twice as accurate as Clenshaw-Curtis. 
 
\begin{figure}[htbp]
  \centering
  	\subfigure
	[Example 1]
		{\includegraphics[width=6.3cm]{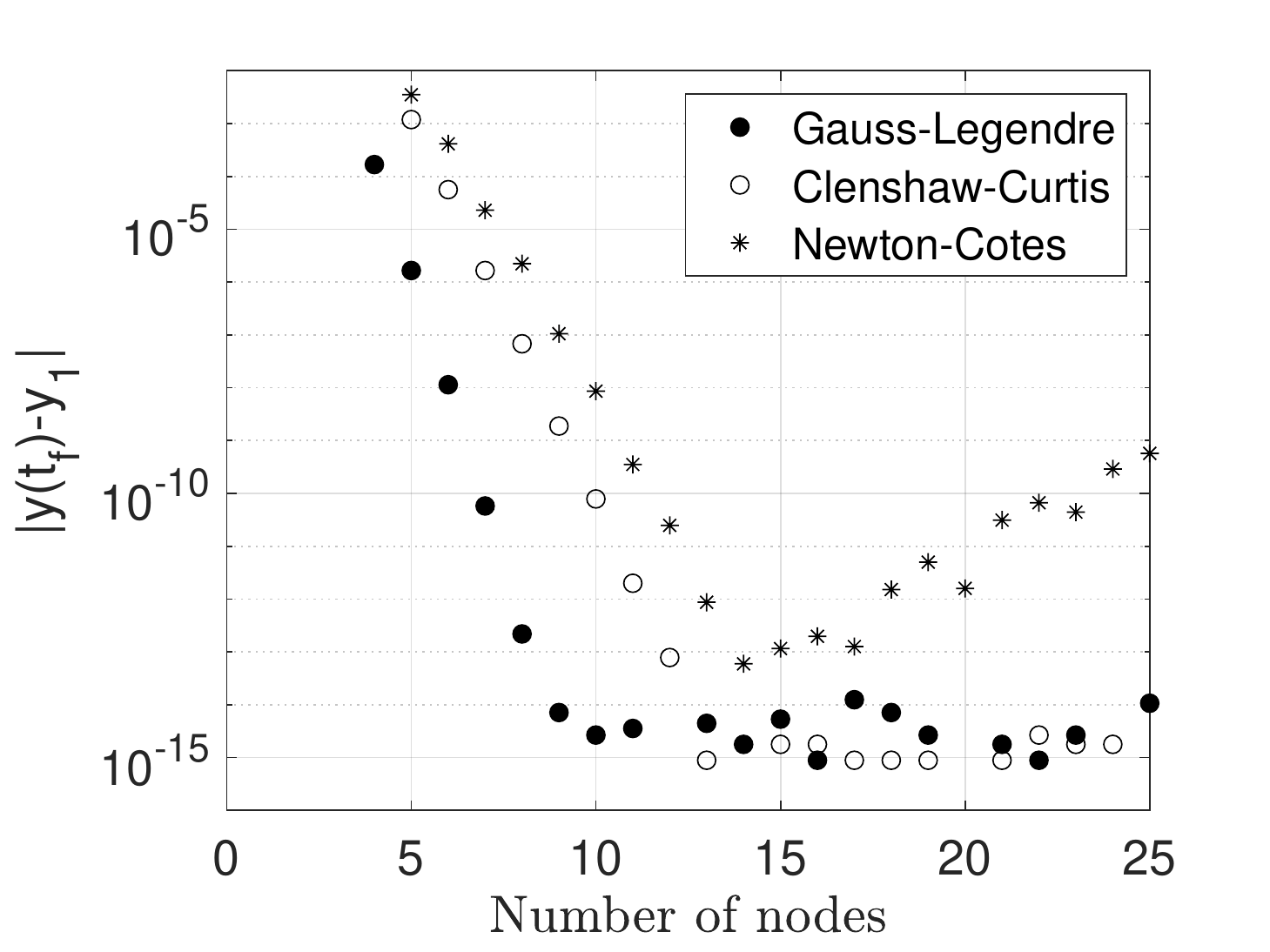}}
		\hfill
		\subfigure[Example 2]
		{\includegraphics[width=6.3cm]{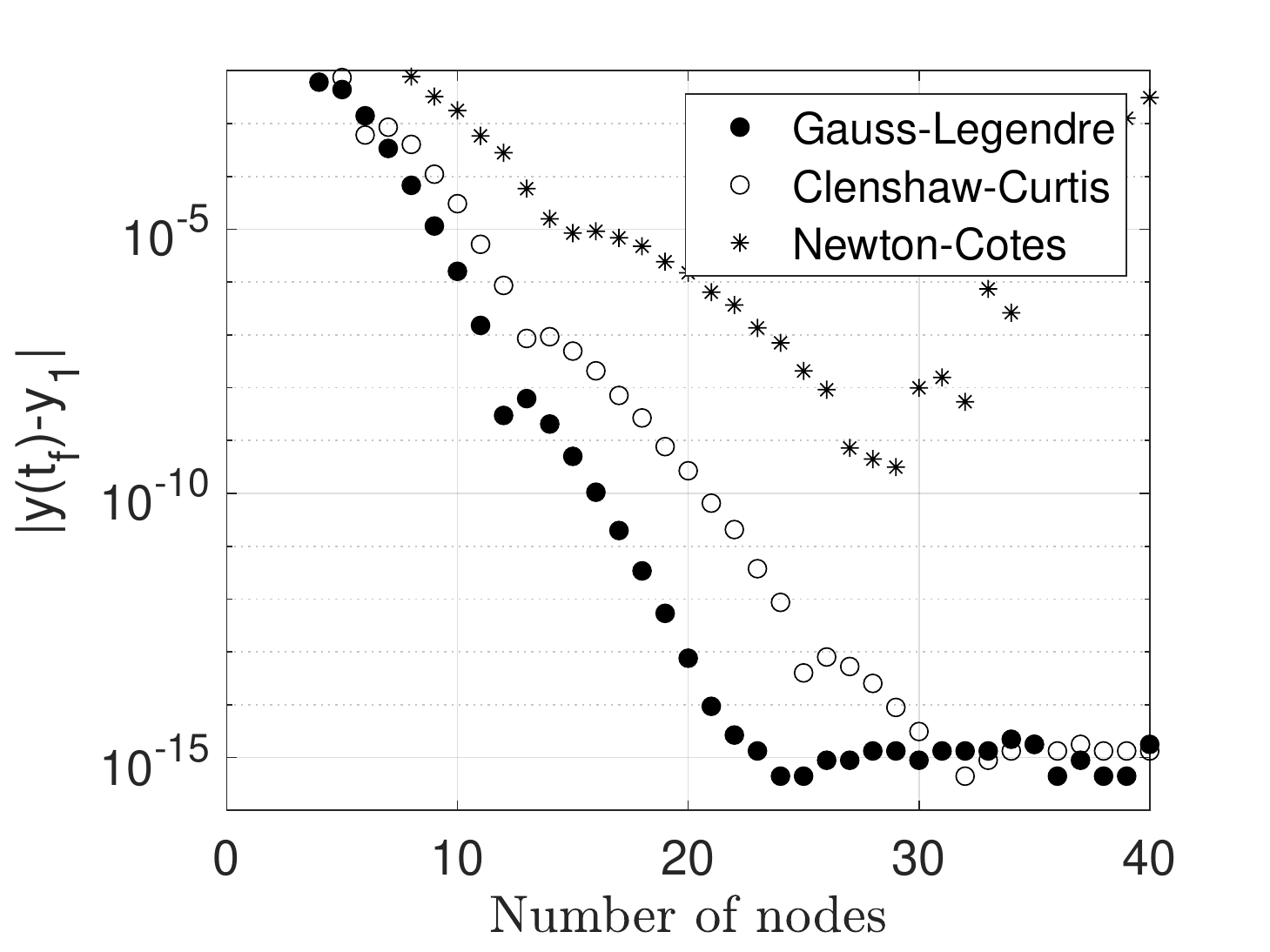}\label{fig:GLvsCC2}}
		\hfill
  \caption{Accuracy of Gauss-Legendre, Clenshaw-Curtis, and Newton-Cotes collocation methods for Examples 1 and 2. }
  \label{fig:GLvsCC}
\end{figure}

The order of accuracy is a widely accepted metric for the accuracy of numerical integration methods. For instance, Shampine and Watts \cite{shampine1972stable} stated that Netwon-Cotes methods have arbitrarily high orders of accuracy. We now know that these orders are only realized for a limited range of nodes, even for a simple function such as $\dot{y} = y$. What does the order of accuracy mean for collocation methods? If a collocation method has as an order $p$ then the corresponding quadrature method is exact for all polynomials of degrees $<p$. Therefore, the failure of the order to predict the actual accuracy is not surprising but is a direct consequence of the failure of the exactness principle for quadrature methods. We refer the reader to Ref. \cite{trefethen2022exactness} that the failure of this principle
in predicting the actual behavior of quadrature methods in several cases, including including Newton-Cotes and Clenshaw-Curtis quadrature. In Ref. \cite{atallah2022comparative}, we also present more examples that demonstrate the same accuracy of the Gauss-Legendre and Clenshaw-Curtis methods when solving gravity-perturbed motion for objects in the vicinity of the Earth.  


\section{Conclusions}
\label{sec:conclusions}
Gauss-Legendre collocation is the Runge-Kutta family of methods with the highest orders of accuracy, and they are also A-stable for all $s>0$. The family Clenshaw-Curtis collocation methods, on the other hand, have several properties that make them more appropriate for practical implementations. Although they are expected to be half as accurate as Gauss-Legendre ones, numerical examples indicate that both families may be equally accurate. Furthermore, our investigation into the stability of these methods reveals that they are A-stable when $2\leq s \leq 78$, making them suitable for stiff problems.



\bibliographystyle{plain}

\bibliography{references2}

\end{document}